\documentclass
{amsart}
\usepackage{amssymb,mathrsfs}
\usepackage{color,umoline}
\usepackage[dvipsnames]{xcolor}
\usepackage{graphicx}
\usepackage{enumitem}
\usepackage[font=small,labelfont=bf]{caption}
\usepackage{tikz, caption, subcaption}
\usepackage{relsize}
\usepackage[mathscr]{eucal}
\usepackage{verbatim}
\usepackage[draft]{todonotes}
\usepackage[pdftex]{hyperref}
\usepackage[dvipsnames]{xcolor}

\newcommand{\supp}{\operatorname{supp}}

\newcommand{\Be}{\begin{equation}}
\newcommand{\Ee}{\end{equation}}

\numberwithin{equation}{section}

\newtheorem{thm}{Theorem}[section]

\newtheorem{lem}[thm]{Lemma}

\theoremstyle{remark}
\newtheorem{rem}{Remark}

\newcommand{\il}{I(\lambda, \Phi, \psi)}

\title{Uniform stationary phase estimate  \\ with limited smoothness}

\author[S. Lee]{Sanghyuk Lee}
\author[S. Oh]{Sewook  Oh}

\address{Department of Mathematical Sciences and RIM, Seoul National University, Seoul 08826, Republic of Korea}
\email{shklee@snu.ac.kr}

\address{School of Mathematics, Korea Institute for Advanced Study, Seoul 02455, Republic of Korea}
\email{sewookoh@kias.re.kr}


\begin{document}
\begin{abstract}  
 In this paper, we consider the uniform estimate for the oscillatory integral with stationary phase, which was previously studied by Alazard--Burq--Zuily.  We significantly reduce the order of required regularity condition on  the phase and amplitude functions for the uniform  estimate.  We also study estimates for the oscillatory integrals of which phase and amplitude functions depend on the oscillation parameter.    The novelty of this article lies in  the use of the wave packet decomposition, which transforms the decay estimate for the oscillatory integral to the disjointness property of the supports of wave packets.   The latter  is  geometric in its nature and less sensitive to the smoothness of the phase and amplitude functions.  
 \end{abstract}

\maketitle

\section{Introduction}

Let $\psi \in C^\infty_0(\mathbb R^d)$ and  $\Phi$ be smooth  on the support of $\psi$.   For $\lambda\ge 1$,  we set 
\begin{equation}
\label{def1}
I(\lambda, \Phi, \psi)=\int e^{i\lambda\Phi(x)}\psi(x)dx.
\end{equation}
If $\Phi$ is nondegenerate, that is to say, the hessian matrix of 
$\Phi$ is nonsingular on the support of $\psi$,  the stationary phase method  gives 
\newcommand{\B}{\mathfrak B}
\begin{equation}
\label{stationary}
|\il |\le \B\lambda^{-{d\over 2}}
\end{equation}
with a constant $\B$ depending on $\Phi$. (See, for example, H\"ormander \cite{Ho} or Stein \cite{St}.) 
The stationary phase method is an important tool, which has a wide range of applications. 

There are various occasions where we need the  estimate \eqref{stationary} with $\B$ independent of the phase $\Phi$ and $\psi$ while they  are normalized in a suitable way  (see \cite{BCD, FRT, KPV}).  Besides,  a uniform bound on the stationary phase estimate makes it possible to recover the optimal bound for degenerate phase if a suitable decomposition away from the degeneracy is available  (see, for example, \cite{BLN}).  In one dimension,  thanks to van der Corput's Lemma it is relatively easy to get the uniform estimate as long as we have a lower bound on the second derivative  $|\Phi''|$ (see \cite{St}).
However, in higher dimensions, precise control of the behavior, with respect to the phase and symbol, of the constant $\B$ becomes far less trivial.

\subsection*{Uniform stationary phase estimate} We set $\mathbb B_r=\{ x\in \mathbb R^d: |x|<r\}$. 
Without loss of generality, we may assume that $\supp \psi\subset \mathbb B_1$ and $\Phi$ is smooth on $\mathbb B_1$. For $k\in \mathbb{N}_0$ and $0<\alpha<1$,  let us denote by $C^{k,\alpha}$ the H\"older space. For simplicity we denote $C^{k+\alpha}=C^{k,\alpha}.$ 
And define $\lVert \phi \rVert_{C^{r}(\mathbb B_1)}= \lVert \phi \rVert_{C^{[r]}(\mathbb B_1)} + \sum_{|\alpha|= [r]} |  \mathcal{D}^\alpha \phi|_{C^{0,r-[r]}(\mathbb B_1)}$ where $|f|_{C^{0,\beta}}=\sup_{(x,y)\in \mathbb B_1\times \mathbb B_1: x\neq y} |f(x)-f(y)|/|x-y|^\beta$  for $0<\beta<1$. 
For $r\ge 2$, we define
\begin{align*}
\mathcal{P}_r:= \sum_{|\beta|=2}\lVert \mathcal{D}^\beta \Phi\rVert_{C^{r-2}(\mathbb B_1)} , \  \ \mathcal{A}_r:= \lVert \psi\rVert_{C^{r}(\mathbb B_1)},
\end{align*}
which respectively control the phase and the amplitude functions.   
Note that $\mathcal{P}_r$ is determined by derivatives of degree bigger than or equal to $2$.  We also set  \[ L^\ast (\Phi)=\inf_{x\in \mathbb B_1} |(\det \mathrm{H} \Phi(x))|.\]
The following is due to  Alazard--Burq--Zuily \cite{ABZ1}. 

\begin{thm}[{\cite[Theorem 3]{ABZ1}}]
\label{thm:ABZ1} 
 Let $\Phi \in C^{d+2}(\mathbb B_1)$ and $\psi \in C_c^{d+1}(\mathbb B_{1-\epsilon_0})$ for some small constant $\epsilon_0>0$. Suppose that 
$\mathbb B_1  \ni x \mapsto \nabla\Phi(x)$  is injective  and $ L^\ast (\Phi)>0$,   then we have
\begin{equation}
\label{ABZ1}
\lambda^{{d\over 2}}|I(\lambda, \Phi, \psi)|\le {C\over  {\  L^\ast (\Phi)}} (1+\mathcal{P}_{d+2}^{d \over 2})\mathcal{A}_{d+1}
\end{equation}
with $C$ depending  only on $d$ and $\epsilon_0$.
\end{thm}


The regularity assumption  
in Theorem \ref{thm:ABZ1}  seems to be rather too strong and unlikely to be optimal.  
In this note, we attempt to relax the  regularity condition for 
the uniform stationary phase estimate.  
However, 
even for the  phase 
without a stationary point, i.e., $\nabla \Phi\neq 0$ on the support of $\psi$,  in order to get the decay of order $\lambda^{-\frac d2}$ (\eqref{stationary}), it is necessary to carry out integration by parts at least $d/2$ times, so  we need to assume that
$\Phi, \psi\in C^k$, $k\ge d/2$.   
In this regard, it is natural  to expect that a threshold should be bigger than $d/2$ for the uniform stationary phase estimate to hold.

The following is our first result which   significantly improves the regularity requirement in Theorem \ref{thm:ABZ1}. It is remarkable that  this brings the order of regularity close to the threshold $d/2$.

\begin{thm}\label{PS0}  
Let $n_0> {{(d+2)}/{2}}$. Let $\Phi \in C^{n_0}(\mathbb B_1)$ and $\psi \in C^{[{{d}/{2}}]+1}_c(\mathbb B_{1-\epsilon_0})$  for some small constant $\epsilon_0>0$.  Suppose that 
$\mathbb B_1\ni x \mapsto \nabla\Phi(x)$  is injective  and $ L^\ast (\Phi)>0$,  then there is a constant $C$, depending only on $n_0,\epsilon_0$, and $d$,   such that  
\begin{align}
\label{main-est}
\lambda^{{d\over2}}|I(\lambda, \Phi, \psi)| \le  {C \over  L^\ast (\Phi)}(1+\mathcal{P}_{n_0}^{\frac d2}) \mathcal{A}_{[{{d}\over{2}}]+1} . 
\end{align}
\end{thm}

The approach in \cite{ABZ1} relies  on integration by parts  with respect to
the operator $\mathcal{L}= |\nabla\Phi|^{-2 }{\nabla\Phi \cdot \nabla }$ to control a part of the integral $I(\lambda, \Phi, \psi)$ which is away from the critical point.  This requires 
extra order of differentiability for phase and amplitude functions. 
To avoid this,  we take a different approach based on the wave packet decomposition which has been widely used in the study of the restriction problem.  The localization in the phase space gives a significant advantage in controlling the phase function under weaker regularity assumptions.  The primary advantage of wave packet decomposition is to reduce 
 the estimate for the oscillatory integral to a problem of estimating a geometric quantity, more precisely, the degree of overlap between the supports of the wave packets (e.g., Section \ref{sec:wpd}).

\begin{rem}  \label{rem1}
If we remove the injectivity assumption in Theorem \ref{PS0}, using \cite[Lemma 7]{ABZ1}, one can easily show  that
\begin{align*}
\lambda^{{d\over2}}|I(\lambda, \Phi, \psi)| \le  {C \over  (L^\ast (\Phi))^{d+1}}(1+\mathcal{P}_{n_0}^{2d+2}) \mathcal{A}_{[{{d}\over{2}}]+1} . 
\end{align*}
\end{rem}

\subsection*{Oscillatory integral depending on $\lambda$} 
Now we consider the oscillatory integrals of which phase and amplitude functions are no longer independent of the oscillatory parameter $\lambda$.  In such cases, the oscillatory integral estimate becomes more involved than the case where the phase and amplitude are independent of $\lambda$.  In practice, those kinds of oscillatory integrals  typically appear as results of applying cutoff functions depending on $\lambda$.  

Let us set 
\[  {\widetilde I} (\lambda):= I(\lambda, \Phi_\lambda, \psi_\lambda).\]  
Here, we use the notation $ \Phi_\lambda$ and  $ \psi_\lambda$ to indicate that the phase and amplitude functions depend on 
$\lambda$.  We consider  ${\widetilde I} (\lambda)$ under  the following conditions: 
  \begin{align}
\label{LSA-psi}
 |\mathcal{D}_x^{\gamma}\psi_\lambda(x)|&\le\tilde{\mathcal  A} \lambda^{|\gamma|\beta}, \quad\quad \ \ 0\le|\gamma|\le d+1,
 \\
\label{LSA-Phi1}
|\mathcal{D}_x^{\gamma}\Phi_\lambda(x)|&\le\tilde{\mathcal P}\lambda^{\beta (|\gamma|-2)}, \quad  2\le|\gamma|\le d+1,
\end{align}
for $x \in \mathbb B_1$ and  some positive constants  $\tilde{\mathcal  A}$ and $\tilde {\mathcal P}$. 
A weaker decay in $\lambda$ is naturally expected by the presence of extra power of $\lambda$ in \eqref{LSA-psi} and \eqref{LSA-Phi1}. However, 
for $0<\beta\le 1/2$, it is easy to show  the best possible decay ${\widetilde I} (\lambda)=O(\lambda^{-d/2})$, for example,  by adapting the argument used to prove  Theorem \ref{PS0} (see Section \ref{sec:pa}). If $\beta$ in \eqref{LSA-psi} and \eqref{LSA-Phi1} is bigger than $1/2$, then  the bound  $O(\lambda^{-d/2})$ is no longer available (see Remark  \ref{remark3} below).  Tacy \cite{M}  showed that $ {\widetilde I} (\lambda)\le C L^\ast (\Phi_\lambda)^{-1}\lambda^{-d(1-\beta)}$ under a restrictive  assumption on the phase.  More precisely,  it was assumed in \cite{M} that there exist $\mu\ge \lambda^{-1+\beta}$ and $C_\gamma>0$ such that
 all eigenvalues of $ \mathrm{H}\Phi_\lambda(x_c)$ are in $[ \mu,2\mu]$ and  $|\mathcal{D}_x^{\gamma}\Phi_\lambda(x)|\le C_\gamma\mu$ for all multi-indices $\gamma$. Here, $x_c$ denotes the critical point of $\Phi$. 

In what follows we show that a  similar result remains valid in a more general and natural setting.

\begin{thm}\label{LSE}
Let $1/2 \le\beta< 1$ and $0<\epsilon_0\ll 1$. 
Let $ \Phi_\lambda\in C^{d+1}(\mathbb B_1)$ and $\psi_\lambda\in C^{d+1}_c(\mathbb B_{1-\epsilon_0})$ satisfy \eqref{LSA-psi} and \eqref{LSA-Phi1}.  
Suppose that  $\mathbb B_1 \ni x  \mapsto \nabla_x\Phi_\lambda(x)$ is injective and $ L^\ast (\Phi_\lambda)>0$.  Then there exists a constant $C$ such that
\Be 
\label{i-lambda}
|{\widetilde I}(\lambda)|\le  \frac C{L^\ast (\Phi_\lambda)}(1+\tilde {\mathcal P}^{d\beta})\tilde{\mathcal  A}\,\lambda^{-d(1-\beta)}.
\Ee 
\end{thm}


\begin{rem}
\label{remark3}	
When $1/2\le\beta\le1$, the decay rate $d(1-\beta)$ in \eqref{i-lambda} is optimal.  To see this,  set $\Phi_\lambda(x)=|x|^2$ and  $\alpha=2\beta-1$. We consider 
 \[\psi_\lambda(x)=\sum_{\nu\in\lambda^{-\alpha}\mathbb{Z}^d\cap \mathbb B_{1/2}}\phi(2\lambda^{\alpha}(\lambda^{(1-\alpha)/2}x-\nu))e^{-i\lambda^\alpha|\nu|^2},\]
where $\phi\in C^\infty_c(\mathbb B_1)$ is nonnegative and $\phi\ge1$ on $\mathbb B_{1/2}$.
Then, it is easy to see that $\psi$ satisfies \eqref{LSA-psi} with $\beta=(\alpha+1)/2$. Set  
\[I_{\lambda,\nu}(\lambda)=\int e^{i\lambda^\alpha(|\lambda^{(1-\alpha)/2}x|^2-|\nu|^2)}\phi(2\lambda^{\alpha}(\lambda^{(1-\alpha)/2}x-\nu))dx.\]  
Hence,  we have ${\widetilde I}(\lambda)=\sum_\nu I_{\lambda,\nu}(\lambda)$, and 
$I_{\lambda,\nu}(\lambda)=\lambda^{-d(1+\alpha)/2}\int e^{ix\cdot(2\nu+\lambda^{-\alpha}x)}\phi(2x)dx$  changing  variables. Since $|x|,|\nu|\le 1/2$, $Re(I_{\lambda,\nu}(\lambda))\ge C^{-1} \lambda^{-d(1+\alpha)/2}$ for some constant $C>0$. This gives  $|{\widetilde I}(\lambda)|\gtrsim \lambda^{-d(1-\beta)}$ as desired.
\end{rem}

\begin{rem}\label{rem4}
	If  $0<\beta<1/2$, we can get the same result under a weaker regularity assumption in a similar manner as in Theorem \ref{PS0}. Precisely, we only need \eqref{LSA-psi} for $0\le |\gamma| \le [d/({2-2\beta})]+1$, and \eqref{LSA-Phi1} for $2\le |\gamma| \le [ (d-2)/({2-2\beta})]+3$.  This can be shown by modifying the proof of Theorem \ref{PS0}. Instead of providing the detail  we leave its proof to the interested reader. 
\end{rem}

\subsection*{Notation} For positive $A,B$ numbers,  we say $A\lesssim B$ if there exists $C$ depending only on $d,\epsilon_0,n_0,\beta$ such that $A\le CB$. 

\section{Uniform stationary phase estimates}

\subsection{Wave packet decomposition}\label{sec:wpd}
We begin with providing an alternative proof of  the result in \cite{ABZ1}, making use of the wave packet decomposition.  In fact, we show 
\begin{align}
\label{abz-result}
\lambda^{{d \over 2}}|I(\lambda, \Phi, \psi)| 
&\lesssim  L^\ast (\Phi)^{-1}(1+\mathcal{P}_{d+1}^{2d+1}) \mathcal{A}_{d+1}.
\end{align}
Note that  $I(\lambda,\Phi,\psi)=I(t\lambda,t^{-1}\Phi,\psi)$. Taking $t=1+\mathcal{P}_{d+1}$, we can recover the result of Alazard--Burq--Zuily, that is to say,  Theorem 
\ref{thm:ABZ1}.   More precisely, we get a slightly stronger result since we need only to control   derivatives of the phase and amplitude up to $(d+1)$-th  order.

Let  $\phi$ be a smooth nonnegative function such that $\sum_{k \in \mathbb{Z}^d} \phi(x-k)=1$ and $\supp \phi \subset \mathbb B_{{1}}.$ For $\nu\in \lambda^{-{1/2}}\mathbb{Z}^d,$ we set 
\[I(\lambda,\nu)= \int e^{i\lambda\Phi(x)}\psi(x) \phi(\lambda^{1/2}(x-\nu)) dx.\] 
Note that $I(\lambda,\nu)\neq 0$ only if $ \nu \in  \supp \psi+ O(\lambda^{-1/2})$. 
Since $\supp \psi \subset \mathbb B_{1-\epsilon_0}$, for $\lambda\ge  10/\epsilon_0^2$ we  decompose
\begin{align}
\label{wp-decomp}
I(\lambda, \Phi, \psi)=\sum_{\nu \in \mathcal L_\lambda}  I(\lambda,\nu), 
\end{align}
where $\mathcal L_\lambda =\lambda^{-{1/2}}\mathbb{Z}^d \cap \mathbb B_{1}$.  
For  $\nu\in \mathcal L_\lambda$,  we set 
\begin{align*}
\Phi^{\lambda, \nu}(x)&=\Phi(\lambda^{-{1 \over 2}}x + \nu)-\Phi(\nu)-\lambda^{-{1 \over 2}}\nabla\Phi(\nu) \cdot x,
\\
\chi^{\lambda, \nu}(x)&=e^{i\lambda\Phi^{\lambda, \nu}(x)}\psi(\lambda^{-{1 \over 2}}x + \nu)\tilde{\phi}(x), 
\end{align*}
where $\tilde{\phi} \in C^{\infty}_{0}$ is fixed nonnegative function satisfying $\tilde{\phi}\equiv1$ on $\mathbb B_1$ and $supp \ \tilde{\phi} \subset \mathbb B_{2}$.
Changing  variables $x\to \lambda^{-{1/ 2}}x + \nu$,  we have 
\Be
\label{il}
I(\lambda, \nu) =\lambda^{-{d \over 2}}
e^{i\lambda\Phi( \nu)} \int e^{i\lambda^{{1 \over 2}}\nabla\Phi(\nu) \cdot x} \chi^{\lambda, \nu}(x)dx.
\Ee

By the mean value theorem we have
\begin{align}
\label{dphi1}
 |\nabla (\lambda \Phi^{\lambda, \nu})(x)|
                   &\lesssim \mathcal{P}_2,  \quad x \in \supp \chi^{\lambda, \nu} . 
 \end{align}
A routine  computation shows 
 \begin{align}
 \label{dphi2}
\| \mathcal{D}_x^{\gamma}(\lambda\Phi^{\lambda, \nu})\|_{C(\mathbb B_{2})}
                  & \lesssim \mathcal{P}_{|\gamma|} \lambda^{-{\frac12}(|\gamma|-2)},
\\
\label{dphi3}
\lVert \mathcal{D}_x^{\gamma}(\lambda\Phi^{\lambda, \nu})\lVert_{C^{\beta}(\mathbb B_{2})}
                                    &\lesssim \mathcal{P}_{|\gamma|+\beta}\lambda^{-{\frac12}(|\gamma|-2)}
\end{align} 
 for $|\gamma|\ge 2$ and  $0<\beta<1$. 
 If  $\Phi$ and $\psi\in C^{d+1}$, by \eqref{dphi1} and \eqref{dphi2} we have 
\Be
\label{dchi}  \|\chi^{\lambda, \nu}\|_{C^{d+1}}  \lesssim(1+\mathcal{P}_{d+1}^{d+1}) \mathcal{A}_{d+1}. 
\Ee
By integration by parts $(d+1)$-times, one can easily see that 
$  |I(\lambda, \nu)|\lesssim(1+\mathcal{P}_{d+1}^{d+1}) \mathcal{A}_{d+1} (1+\lambda^{{1 \over 2}}|\nabla\Phi(\nu)|)^{-d-1}$. 
Hence, 
\begin{align*}
 \lambda^{{d \over 2}}|I(\lambda, \Phi, \psi)| &\lesssim  (1+\mathcal{P}_{d+1}^{d+1}) \mathcal{A}_{d+1} \sum _{\nu \in \mathcal L_\lambda}  (1+\lambda^{{1 \over 2}}|\nabla\Phi(\nu)|)^{-d-1}.
\end{align*}
It is now sufficient for  \eqref{abz-result} to show the  following lemma, which is  in a slightly more general form than we need it here.

\begin{lem}\label{k-sum}  Let $k \in \mathbb R^d$ and $\lambda\ge 10/\epsilon_0^2$.
Suppose that $ \mathbb B_1\ni x \mapsto \nabla\Phi(x)$  is injective and $ L^\ast (\Phi)>0$. Then, for a constant $C$  we have 
\begin{equation*} 
\sum _{\nu \in \mathcal L_\lambda}
(1+|k+\lambda^{{1 \over 2}}\nabla\Phi(\nu)|)^{-d-1}\le C  L^\ast (\Phi)^{-1}(1+\mathcal{P}_2^d).
\end{equation*}
\end{lem}

\begin{proof}
[Proof  of  Lemma \ref{k-sum}] We split the sum $ \sum _{\nu \in \mathcal L_\lambda}  (1+\lambda^{{1 \over 2}}|\nabla\Phi(\nu)|)^{-d-1}$ into two parts $\rm I$ and $\rm I\!I$ which are given by 
\begin{align*}
 \mathrm I&=\sum _{\nu \in \mathcal L_\lambda: |k+\lambda^{{1 / 2}}\nabla\Phi(\nu)|\ge\sqrt{d}\mathcal{P}_2 }   
\big(1+|k+\lambda^{{1 \over 2}}\nabla\Phi(\nu)|\big)^{-d-1}, 
 \\ 
 \mathrm I\! \mathrm I&= \sum _{\nu \in \mathcal L_\lambda:  |k+\lambda^{{1/ 2}}\nabla\Phi(\nu)|<\sqrt{d}\mathcal{P}_2 }  
\big(1+|k+\lambda^{{1 \over 2}}\nabla\Phi(\nu)|\big)^{-d-1}.  
\end{align*}
If $|k+\lambda^{{1 \over 2}}\nabla\Phi(\nu)|\ge\sqrt{d}\mathcal{P}_2$,  by the mean value inequality  we have 
$|k+\lambda^{{1 \over 2}}\nabla\Phi(\nu)|\ge 2^{-1}|k+\lambda^{{1 \over 2}}\nabla\Phi(x)|$ for $x\in  Q_\nu:= \nu + \lambda^{-\frac12}[-{1/2},{1/2}]^d$. 
Thus, 
\[ 
\big(1+|k+\lambda^{{1 \over 2}}\nabla\Phi(\nu)|\big)^{-d-1}
\le C\lambda^{\frac d2} \int_{Q_\nu} (1+|k+\lambda^{{1 \over 2}}\nabla\Phi(x)|)^{-d-1}dx.\]
Consequently, we have 
\[ \mathrm I\le C_d\lambda^{{d \over 2}}
\int_{
\mathbb B_1
 } (1+|k+\lambda^{{1 \over 2}}\nabla\Phi(x)|)^{-d-1}dx\]
for some constant $C_d$. 
Since $\mathbb B_1\ni x \mapsto \nabla\Phi(x)$  is injective   and since $ L^\ast (\Phi)>0$, changing  variables $x \mapsto (\nabla\Phi)^{-1} (x)$,  we get
\begin{equation*}
\mathrm I\lesssim\int_{\nabla \Phi(\mathbb B_1)}(1+|k+y|)^{-d-1}|\det \mathrm{H} \Phi( (\nabla\Phi)^{-1} (y))|^{-1}dy\lesssim L^\ast (\Phi)^{-1}.
\end{equation*}

Now we consider the sum  $\rm I\!I$. Using $(1+|k+\lambda^{{1 \over 2}}\nabla\Phi(\nu)|)^{-d-1}\le 1$ and following the same argument, we see ${\rm I\!I}$ is bounded above by       
\begin{align*}
  \lambda^{\frac d2} \int_{x\in \mathbb B_1: |k+\lambda^\frac12 \nabla\phi(x)|\le 2\sqrt{d}\mathcal{P}_2} dx  
                  \lesssim 
                   \int_{y\in \nabla \Phi(\mathbb B_1): |k+y|\le  2\sqrt{d}\mathcal{P}_2}|\det \mathrm{H} \Phi( (\nabla\Phi)^{-1} (y))|^{-1}dy.\end{align*}
                  The last is clearly bounded by $\lesssim L^\ast (\Phi)^{-1}\mathcal{P}_2^d.
$
\end{proof}

\subsection{Estimate with reduced regularity}  In the above we have observed  that the wave packet decomposition provides an easy way to prove a uniform 
stationary phase estimate $\eqref{ABZ1}$. We elaborate the argument to lower  required  order of  regularity. 
To improve the preliminary result, we expand  $\chi^{\lambda, \nu}$ into Fourier series which gives  rise to another summation. In order to guarantee  absolute summability of the series we  need to show the Fourier coefficients of ${\chi^{\lambda, \nu}}$ decay fast enough, i.e.,  $(\chi^{\lambda, \nu})^\wedge(k)=O(|k|^{-d-\epsilon})$ for some $\epsilon>0$. However, in view of the typical (integration by parts) argument,  we need to differentiate more than $d$-times, hence we are forced to assume differentiability of order higher than $d$. To get over this,  we  examine   $\chi^{\lambda, \nu}$ carefully.

  Recalling that $\chi^{\lambda, \nu}$ is supported in $\mathbb B_{2}$, for $k\in \mathbb N_0^d$ we set 
\[C^{\lambda, \nu}_k =(2\pi)^{-d}\int e^{-ik \cdot x} \chi^{\lambda, \nu}(x)dx.\] 
Using  \eqref{il},  we have 
\begin{equation}
\label{decompfc}
\lambda^{{d \over 2}} I(\lambda, \Phi, \psi)= \sum _{\nu \in \mathcal L_\lambda} 
 \sum_{k\in \mathbb{Z}^d} C^{\lambda, \nu}_k\int e^{i(\lambda^{{1 \over 2}}\nabla\Phi(\nu)+k) \cdot x} \phi(x)\, dx.
\end{equation} 
 The last integral equals $\widehat \phi(\lambda^{{1 \over 2}}\nabla\Phi(\nu)+k)$ which vanishes rapidly if $|k|\gg\lambda^\frac12$.
Thus, the main contribution comes from $|k|\lesssim  \lambda^{1/2}$. Under such a condition, we have the following:

\begin{lem}\label{EFC1}
Let $n_0>(d+2)/2$.  Let $\Phi \in C^{n_0}(\mathbb B_1)$ and $\psi \in C^{[{{d}/{2}}]+1}_c(\mathbb B_{1-\epsilon_0})$.  If   $1\le |k|\le 5\max\{\mathcal P_2,1\}\lambda^{1 \over 2}$, then  there is a constant $C>0$, independent of $\lambda, \nu$,  such that 
\begin{align}
\label{cln}
|C^{\lambda, \nu}_k| \le C \mathcal{A}_{[{{d}\over{2}}]+1}(1+\mathcal{P}_{n_0}^{d+2})|k|^{-(d+\epsilon)}
\end{align}
for some $\epsilon>0$. 
\end{lem}

Compared with the typical argument using integration by parts, the desired decay of Fourier coefficients is obtained under a weaker regularity assumption.  
The key observation is that differentiations on $\chi^{\lambda, \nu}$ produce extra factors of $\lambda^{-\frac12}$ and  these  factors can be converted to get additional decay   because $|k|\lesssim \lambda^\frac12$. 

In order to exploit the extra regularity given by H\"older continuity, we make use of the following lemma which allows us to deal with functions of holder continuity of order $\alpha$ as if they were $\alpha$ times differentiable. 

\begin{lem}\label{HOLDER}
Let $n$ be an nonnegative integer and $0\le\beta<1$. If  $f \in C^{n+\beta}(\mathbb B_2)$ and $g\in C_c^{n+1+\beta}(\mathbb B_2)$, then we have
\begin{align}
\label{1beta}
\left|\int f(\lambda^{-p}x) g(x)e^{-ik\cdot x}dx\right| &\lesssim \frac{(\lambda^{-p\beta}+|k|^{-1})}{ |k|^{n+\beta}}  \lVert f\lVert_{C^{n+\beta}(\mathbb B_{2})} 
\lVert g\lVert_{C^{n+1+\beta}(\mathbb B_{2})}.
\end{align}
with the implicit constant independent of $p\ge0$, $\lambda>1$.
\end{lem}

\begin{proof}
[Proof  of  Lemma \ref{HOLDER}]
We begin with recalling the well-known estimate  
\Be \label{0beta}
\left|\int h(x) e^{-ik\cdot x}dx\right| \lesssim \lVert h\lVert_{C^{\beta}(\mathbb B_2)} 
|k|^{-\beta}
\Ee
for $0\le\beta<1$ whenever $h\in C^{\beta}_c(\mathbb B_2)$.  After integration by parts $n$-times using 
\[\mathcal{D}_{(k)}:={k \over |k|}\cdot \nabla,\] we need only to  show that \eqref{1beta} with $n=0$.
Using $2e^{-ik\cdot x}=e^{-ik\cdot x}-e^{-ik\cdot (x+{\pi k\over|k|^2})}$ 
and change of variables $x\to x+ {\pi k\over|k|^2}$, we see that $\int f(\lambda^{-{p}}x) g(x)e^{-ik\cdot x}dx$ equals
\begin{align*}
{\frac 12}\int \Big(f(\lambda^{-p}x) g(x)-f(\lambda^{-{p}}(x-{\pi k\over|k|^2})) g(x-{\pi k\over|k|^2})\Big)e^{-ik\cdot x}dx.
\end{align*}
We break 
$
\int f(\lambda^{-p}x) g(x)e^{-ik\cdot x}dx=I+I\!I,  
$
where 
\begin{align*}
I&=\int \Big(f(\lambda^{-p}x)-f(\lambda^{-p}(x-{\pi k\over|k|^2}))\Big) g(x)e^{-ik\cdot x}dx,\\
I\!I&=\int f (\lambda^{-p}(x-{\pi k\over|k|^2})) \Big(g(x)-g(x-{\pi k\over|k|^2})\Big)e^{-ik\cdot x}dx.
\end{align*}
 Since $f\in C^{\beta}$, it is clear that $|I|\le C\lVert f\lVert_{C^{\beta}} \lVert g\lVert_{C^{0}}\lambda^{-p\beta}|k|^{-\beta}$. For $I\!I$, we write 
\[g(x)- g(x-{\pi k\over|k|^2})=\int_{0}^{{\pi \over|k|}} {k \over |k|}\cdot \nabla g(x-\tau{k\over|k|}) d\tau.\]
Thus,
\[ I\! I= \int_{0}^{{\pi \over|k|}} \int f(\lambda^{-p}(x-{\pi k\over|k|^2}))\, {k \over |k|}\cdot \nabla g(x-\tau{k\over|k|})\,  e^{-ik\cdot x}dx \,d\tau . \]
Since $f$ and $\nabla g \in C^{\beta}$,  $\| f(\lambda^{-p}(\cdot-{\pi k\over|k|^2})) {k \over |k|}\cdot \nabla g(\cdot-\tau{k\over|k|})\|_{ C^{\beta}} \lesssim  \|f \|_{ C^{\beta}}\|g\|_{ C^{1+\beta}}$. So, using \eqref{0beta},  the inner integral  is bounded by $|k|^{-\beta}\|f\|_{C^{\beta}}\|g\|_{C^{1+\beta}}$ and, hence, we see that $|I\!I|\lesssim |k|^{-1-\beta}\|f\|_{C^{\beta}}\|g\|_{C^{1+\beta}}$.
\end{proof}

We prove  Lemma \ref{EFC1} by making use of  Lemma \ref{HOLDER}. To show \eqref{cln}, we exploit  the factor $\lambda^{-1/2}$ which is generated by differentiation. Since $n_0> (d+2)/2$, we may assume 
\Be 
\label{n0}
n_0= \frac{d+2} 2+ \alpha
\Ee
for some $0<\alpha<1/2$. 

\begin{proof}[Proof  of  Lemma \ref{EFC1}] 
We  write
\begin{align*}
(\mathcal{D}_{(k)})^{[{{d}\over{2}}]+1 }\chi^{\lambda, \nu}(x)=\sum_{m_1+m_2+m_3=[{{d}\over{2}}]+1} f_{m_1}(x)g_{m_2}(x)h_{m_3}(x),
\end{align*}
where 
\begin{align*}
  f_{m}(x)&=\mathcal{D}_{(k)}^{m}\tilde{\phi}(x), \\
   g_m(x)&=\mathcal{D}_{(k)}^{m}(\psi(\lambda^{-{1/2}}x+\nu)), \\
    h_m(x)&=\mathcal{D}_{(k)}^{m}(e^{i\lambda\Phi^{\lambda, \nu}(x)}). 
  \end{align*}
Applying integration  by parts $[{{d}/{2}}]+1$ times, we have 
\begin{equation}
\label{ckest}
|C_k^{\lambda, \nu}|\lesssim |k|^{-[{{d}\over{2}}]-1} \sum_{m_1+m_2+m_3=[{{d}\over{2}}]+1} \Big|\int e^{-ik\cdot x}f_{m_1}(x)g_{m_2}(x)h_{m_3}(x)dx\Big|.
\end{equation}

 In order to show \eqref{cln}, we further apply integration by parts. However, we need to examine  those functions to ensure uniform bounds.   The factors $f_m$ and  $g_m$ are easier to handle. Note that  $\lVert f_{m}\rVert_{C^{d+1}}\lesssim1$ and
 \Be 
 \label{gm}
 \lVert g_{m}(\lambda^{\frac12}\cdot)\rVert_{C^{[{{d}\over{2}}]+1-m}}\lesssim \mathcal{A}_{[{{d}\over{2}}]+1}\lambda ^{-\frac m2}
 \Ee 
 for any $0\le m\le [{{d}/{2}}]+1$. 
 We first consider the case  $m_3=0$ and claim 
 \begin{equation}
 \label{m30}
|\int e^{-ik\cdot x}f_{m_1}g_{m_2}h_{0}dx|\lesssim \mathcal{A}_{[{{d}\over{2}}]+1}(1+\mathcal{P}_{n_0}^{d+2-m_2})\lambda^{-\frac {m_2} 2}|k|^{-[{{d}\over{2}}]-1+m_2}.
\end{equation}
Being combined with \eqref{ckest}, this shows the contributions of the terms with $m_3=0$ are $O(|k|^{-d-1/2})$ since $|k|\le 5\max\{\mathcal P_2,1\}\lambda^{1 \over 2}$. To show \eqref{m30}, from \eqref{dphi1}--\eqref{dphi3} we note 
$ |\mathcal{D}_{(k)}^{\gamma}  h_0|\lesssim (1+\mathcal{P}_{n_0}^{d+2-m_2}) $ if $\gamma\le [d/2]+1-m_2$. 
Thus, via integration by parts $[d/2]+1-m_2$ times, using \eqref{gm}   we  get \eqref{m30}.

For the rest of the proof, we assume $m_3\ge 1$.  For $m\ge 1$, we write $h_{m}$ as
\begin{align*}
h_{m}(x)=\sum_{n=1}^{m}  \Big(\sum_{i_1+ \cdots +i_n=m, \, 1\le  i_n\le  \cdots \le i_1} C_{n,i_1, \cdots i_n}\,{h_{m,n}^{i_1,\cdots,i_n}(x)}\Big), 
\end{align*}
where $C_{n,i_1, \cdots i_n}$ is a constant and 
\[h_{m,n}^{i_1,\cdots,i_n}(x)=\prod_{l=1}^{n}(\mathcal{D}_{(k)}^{i_l} \lambda\Phi^{\lambda, \nu}(x)) e^{i\lambda\Phi^{\lambda, \nu}(x)}.\] 
Note that  $h_{m,n}^{i_1,\cdots,i_n}\in C^{n_0-i_1}$. Recalling \eqref{n0}, for \eqref{cln} we only need to show  
\begin{equation}
\label{mni}
|\int e^{-ik\cdot x}f_{m_1}g_{m_2}h_{m_3,n}^{i_1,\cdots,i_n}dx|\lesssim \mathcal{A}_{[{{d}\over{2}}]+1}(1+\mathcal{P}_{n_0}^{d+2})|k|^{-(d+2\alpha-[{{d}\over{2}}]-1 )}
\end{equation}
for each fixed $m_1,m_2,m_3,n,$ and $i_1,\cdots,i_n$ which  satisfies  $i_1+ \cdots +i_n=m_3$ and $1\le  i_n\le  \cdots \le i_1$.  Combining this with  \eqref{ckest} immediately shows \eqref{cln}.

In order to show \eqref{mni}, 
we write
   \[ f_{m_1}(x)g_{m_2}(x)h_{m_3,n}^{i_1,\cdots,i_n}(x)=F(\lambda^{-\frac 12}x)G(x),\] 
where 
  \begin{align*}
  F(x)&=g_{m_2}(\lambda^{\frac 12}x)\prod_{i_l\ge2}\lambda^{-\frac {i_l-2}2}(\mathcal{D}_{(k)}^{i_l}\Phi)(x+\nu),
  \\
  G(x)&=f_{m_1}(x)\prod_{i_l=1}(\mathcal{D}_{(k)}^{i_l} \lambda\Phi^{\lambda, \nu}) e^{i\lambda\Phi^{\lambda, \nu}(x)}.
  \end{align*} 
   By Lemma \ref{HOLDER}, it is enough for \eqref{mni} to prove 
\begin{equation}
\label{m1m2m3norm}
\lVert  F \rVert_{C^{n_0-j}} \lVert G\rVert_{C^{n_0-j+1}}\lesssim\mathcal{A}_{[{{d}\over{2}}]+1}(1+\mathcal{P}_{n_0}^{d+3-j})\lambda^{-\frac {j-2}2}
\end{equation}
for  some integer $ j\in [2, n_0)$.  
Indeed, by Lemma \ref{HOLDER} and the assumption $|k|\le 5\max\{\mathcal P_2,1\}\lambda^{1 \over 2}$ we obtain  \eqref{mni}.

It remains to show \eqref{m1m2m3norm}. Since $\Phi \in C^{n_0}(\mathbb B_1)$ and $\psi \in C^{[{{d}/{2}}]+1}_c(\mathbb B_{1-\epsilon_0})$, we note 
 $F\in  C_c^{[d/2]+1-m_2}(\mathbb B_{1}(-\nu))\cap C_c^{n_0-i_1}(\mathbb B_{1}(-\nu))$ and $G\in C^{n_0-1}_c(\mathbb B_{2})$ where $\mathbb B_{1}(-\nu)=\{x\in \mathbb{R}^d:|x+\nu|<1\}$.  We show \eqref{m1m2m3norm} by considering the three cases 
\[ i_1,m_2\le1,  \quad m_2\ge \max\{i_1,2\}, \quad i_1>\max\{m_2,1\},\]
separately.  First,  if $i_1,m_2\le1$, then $F\in C^{n_0-2}(\mathbb B_{2})$. Using 
\eqref{gm} and  \eqref{dphi1}--\eqref{dphi3},  we see that  \eqref{m1m2m3norm} holds with $j=2$.  Secondly, if $m_2\ge \max\{i_1,2\}$, then $F\in C^{[d/2]+1-m_2}(\mathbb B_{2})\subset C^{n_0-m_2-1}(\mathbb B_{2})$.  Similarly, by \eqref{gm} and  \eqref{dphi1}--\eqref{dphi3} it follows that  \eqref{m1m2m3norm} holds with $j=m_2+1$.
Lastly, if $i_1> \max\{m_2,1\}$, then $F\in C^{n_0-i_1}(\mathbb B_{2})$. Since  $F$ includes $\lambda^{-(i_1-2)/2}$ as one of  its factors, using  \eqref{gm} and  \eqref{dphi1}--\eqref{dphi3},  we see that \eqref{m1m2m3norm} holds with $j=i_1$.
\end{proof}

We are now ready to prove Theorem \ref{PS0}.

\begin{proof}
[Proof  of  Theorem \ref{PS0}]   We may assume that $\lambda\ge 10/\epsilon_0^2$.  
We first consider the case $|\nabla \Phi (0)| \ge 2\max\{\mathcal{P}_2,1\}$. In this case, we have $|\eta \cdot \nabla \Phi (x)|\ge \max\{\mathcal{P}_2,1\}$ for $|x|\le1$ where $\eta=|\nabla \Phi (0)|^{-1} \nabla \Phi (0)$. Using the operator $|\eta \cdot \nabla \Phi (x)|^{-1}(\eta\cdot \nabla)$, via integration by parts $([{d/2}]+1)$--times we get 
\[
\left| I(\lambda, \Phi, \psi)\right|\lesssim(\max\{\mathcal{P}_2,1\}\lambda)^{-[{d\over2}]-1}\mathcal{A}_{[{{d}\over{2}}]+1}\big(1+\mathcal{P}_{[{{d}/{2}}]+1}^{[{{d}/{2}}]+1}\big).
\]

Thus, we may now assume 
$|\nabla \Phi (0)| < 2\max\{\mathcal{P}_2,1\}$.  
Recalling \eqref{decompfc}, 
 we split the sum over $k$ into two parts to get  
\begin{align*} 
|\lambda^{{d \over 2}}I(\lambda, \Phi, \psi)|\le  \sum _{\nu \in \mathcal L_\lambda}    \big( A(\nu)+ B(\nu)\big), 
\end{align*}
where 
\begin{align*}
A(\nu):=\sum_{ |k|\ge 5\max\{\mathcal P_2,1\}\lambda^{1 \over 2}} |C^{\lambda, \nu}_k |
|\widehat \phi(\lambda^{{1 \over 2}}\nabla\Phi(\nu)+k)|,  
\\
B(\nu):=\sum_{ |k|< 5\max\{\mathcal P_2,1\}\lambda^{1 \over 2}} |C^{\lambda, \nu}_k |
|\widehat \phi(\lambda^{{1 \over 2}}\nabla\Phi(\nu)+k)|.
\end{align*}
We first deal with  $A(\nu)$. Since $|\nabla \Phi (0)| < 2\max\{\mathcal{P}_2,1\}$, we have $|\lambda^{{1 \over 2}}\nabla\Phi(\nu)|\le 4\max\{\mathcal{P}_2,1\} \lambda^{1 \over 2}$ for $\nu \in \mathcal L_\lambda$. 
Since $|k|\ge 5\max\{\mathcal P_2,1\}\lambda^{1 \over 2}$, we have $|\lambda^{{1 \over 2}}\nabla\Phi(\nu)+k|\sim |k|$. By the rapid decay of $\widehat \phi$ it follows that 
\begin{align*}
|\widehat \phi(\lambda^{{1 \over 2}}\nabla\Phi(\nu)+k)| \lesssim   (1+|\lambda^{{1 \over 2}}\nabla\Phi(\nu)+k|)^{-d-1}|k|^{-d-1}
\end{align*}
for $|k|\ge 5\max\{\mathcal P_2,1\} \lambda^{1 \over 2}$. 
Since $|C^{\lambda, \nu}_k|\le \int |\psi(\lambda^{-{1 \over 2}}x + \nu)\tilde{\phi}(x)|dx \lesssim \mathcal{A}_0$,  we get
\[
A(\nu) \lesssim  
      \mathcal A_0   \sum_{|k|\ge 5\max\{\mathcal P_2,1\}\lambda^{1 \over 2}}  |k|^{-d-1}
             (1+|\lambda^{{1 \over 2}}\nabla\Phi(\nu)+k|)^{-d-1}.\]
Thus, Lemma \ref{k-sum} yields 
\Be 
\label{asum}
 \sum _{\nu \in \mathcal L_\lambda}    A(\nu)   
     \lesssim { L^\ast (\Phi)}^{-1}  \mathcal{A}_0 (1+\mathcal{P}_2^d). 
 \Ee

We now consider $\sum_\nu B(\nu)$. Since $|k|\le 5 \lambda^{1 \over 2}\max\{\mathcal P_2,1\}$,  by \eqref{cln} we have
\[ B(\nu) 
 \lesssim  
      \mathcal \mathcal{A}_{[{{d}\over{2}}]+1}(1+\mathcal{P}_{n_0}^{d+2})  \sum_{|k|< 5\max\{\mathcal P_2,1\}\lambda^{1 \over 2}}    |k|^{-(d+\alpha )}     (1+|\lambda^{{1 \over 2}}\nabla\Phi(\nu)+k|)^{-d-1}.\]
The rest of the argument is the same as before. In fact, using  Lemma \ref{k-sum},  we get     
\begin{align*}
 \sum _{\nu \in \mathcal L_\lambda}    B(\nu) 
&\lesssim { L^\ast (\Phi)}^{-1} \mathcal{A}_{[{{d}\over{2}}]+1}(1+\mathcal{P}_{n_0}^{d+2})  (1+\mathcal{P}_2^d).
\end{align*}
Combining this  and \eqref{asum}, we obtain 
\[\lambda^{\frac d2}|I(\lambda, \Phi, \psi)|\lesssim { L^\ast (\Phi)}^{-1} \mathcal{A}_{[{{d}\over{2}}]+1}(1+\mathcal{P}_{n_0}^{2d+2}).\]
Finally, since $I(\lambda,\Phi,\psi)=I(t\lambda,t^{-1}\Phi,\psi)$, taking $t=1+\mathcal{P}_{n_0}$ we obtain \eqref{main-est}.
\end{proof}

\subsection{Phase and amplitude  depending on the oscillatory parameter $\lambda$}
\label{sec:pa}
From the previous argument we can expect that, as long as the regularity assumption is high enough,  a similar result holds even if $\psi_\lambda(x)$ has some bad behavior in $\lambda$.

 Similarly as before, let us set 
 \begin{align*}  
  \Phi_\lambda^{\lambda, \nu}&:=(\Phi_\lambda)^{\lambda, \nu}, 
  \\ 
    \chi_\lambda ^{\lambda, \nu}(x)&:=e^{i\lambda\Phi_\lambda^{\lambda, \nu}(x)}\psi_\lambda(\lambda^{-{1 \over 2}}x + \nu)\tilde{\phi}(x). 
       \end{align*}
By the same decomposition as in Section \ref{sec:wpd} with $\Phi_\lambda, \psi_\lambda$, we have an analogue  of \eqref{decompfc}:
\Be 
\lambda^{\frac d2}|\mathcal {\widetilde I}(\lambda)| \le  \sum_{k\in\mathbb{Z}^d}\sum_{\nu \in \mathcal L_\lambda } |(\chi_\lambda^{\lambda, \nu})^{\wedge}(k)| | \widehat \phi(\lambda^\frac12\nabla\Phi_\lambda(\nu)+k)|. 
\label{l-sum}
\Ee
 For the case $0\le \beta\le 1/2$, the optimal decay estimate $\lambda^{-d/2}$ is relatively easy to show.  In fact,  the estimate follows
 once we have the desired estimate 
 \Be 
 \label{dl-chi} \|\chi_\lambda^{\lambda, \nu}\|_{C^{d+1}}  \lesssim (1+{\tilde {\mathcal P}}^{d+1}) {\tilde{\mathcal  A}}
 \Ee
for  $0\le\beta\le 1/2$.   From \eqref{LSA-psi} it is clear that 
\Be
\label{psipsi}
|\mathcal{D}_x^{\gamma}(\psi_\lambda(\lambda^{-\frac12}x+\nu))|\le\tilde{\mathcal  A}\lambda^{|\gamma|(\beta-\frac12)}, \quad 0\le |\gamma|\le d+1.
\Ee
From  \eqref{LSA-Phi1}, we also have $|\nabla_x\lambda \Phi_\lambda^{\lambda, \nu}(x)|\le\tilde{\mathcal P}$ and 
\begin{align*}
|\mathcal{D}_x^{\gamma}\lambda \Phi^{\lambda, \nu}_\lambda(x)|
       \lesssim\tilde{\mathcal P} \lambda^{(|\gamma|-2)(\beta-\frac12)}, \quad  2\le |\gamma|\le d+1. 
\end{align*}
 Combining those estimates, one can easily  obtain \eqref{dl-chi}.  
By  \eqref{dl-chi} it follows that $| (\chi_\lambda^{\lambda, \nu})^\wedge(k)|  \lesssim |k|^{-d-1}(1+{\tilde {\mathcal P}}^{d+1}) {\tilde{\mathcal  A}}$. Thus,     \eqref{l-sum} and  Lemma \ref{k-sum} yield
\eqref{i-lambda} for $0<\beta\le 1/2$ as desired.  

Furthermore, refining the above argument in a similar manner as in the proof of Lemma \ref{EFC1}, one can show the statement in Remark \ref{rem4}.

\begin{proof}
[Proof  of  Theorem \ref{LSE}]
The proof  is almost identical to the previous argument for $0<\beta\le 1/2$ except for the estimate for $(\chi_\lambda^{\lambda, \nu})^\wedge(k)$. Indeed, we can reduce the problem to proving a good decay of the Fourier coefficients $(\chi_\lambda^{\lambda, \nu})^\wedge(k)$. 
By \eqref{l-sum} and Lemma \ref{k-sum} with $\Phi$ replaced by $\Phi_\lambda$, we have
\[\lambda^{\frac d2} |{\widetilde I}(\lambda)|\lesssim \frac1{ L^\ast (\Phi_\lambda)}(1+\tilde {\mathcal P}^{d})\sum_{k\in\mathbb{Z}^d}\sup_{\nu}{|(\chi_\lambda^{\lambda, \nu})^\wedge(k)|}.\]

In this case, we can not use the estimate \eqref{dl-chi}.  Instead,   by \eqref{LSA-Phi1}  we have
\[|\mathcal{D}_x^{\gamma}e^{i\lambda(\Phi^{\lambda, \nu}_\lambda(x))}|\lesssim (1+\tilde {\mathcal P}^{d+1})\lambda^{|\gamma|(\beta-\frac 12)}, \quad 0\le|\gamma|\le d+1.\]
Note that the power of $\lambda$ is positive. Combining this and \eqref{psipsi}, by integration by parts $d+1$ times we get 
\begin{equation}
\label{esthat}
|(\chi_\lambda^{\lambda, \nu})^\wedge(k)|\lesssim  (1+\tilde {\mathcal P}^{d+1})\tilde{\mathcal  A}\lambda^{(\beta-\frac12)(d+1)}|k|^{-d-1}.
\end{equation}
 \eqref{esthat} and the trivial estimate $|(\chi_\lambda^{\lambda, \nu})^\wedge(k)|\lesssim\tilde{\mathcal  A}$ give
\begin{align*}
\sum_{k\in\mathbb{Z}^d}\sup_\nu|(\chi_\lambda^{\lambda, \nu})^\wedge(k)|&\lesssim \sum_{|k|\le \lambda^{\beta-\frac12}}\tilde{\mathcal  A}+\sum_{|k|\ge \lambda^{\beta-\frac12}}(1+\tilde {\mathcal P}^{d+1})\tilde{\mathcal  A}\,\lambda^{(\beta-\frac12)(d+1)}|k|^{-d-1}\\
&\lesssim\tilde{\mathcal  A} \lambda^{d(\beta-\frac12)}+(1+\tilde {\mathcal P}^{d+1})\tilde{\mathcal  A}\lambda^{d(\beta-\frac12)}. 
\end{align*}
Thus, we obtain $\lambda^{\frac d2} |{\widetilde I}(\lambda)|\lesssim L^\ast (\Phi_\lambda)^{\!\!-1}(1+\tilde {\mathcal P}^{2d+1})\tilde{\mathcal A}$. Recalling $I(\lambda,\Phi_\lambda,\psi_\lambda)=I(t\lambda,t^{-1}\Phi_\lambda,\psi_\lambda)$ and taking $t=1+\tilde{\mathcal{P}}$,  we see that \eqref{i-lambda} holds. 
\end{proof}

\section*{Acknowledgement.}  This work was supported by the National Research Foundation (Republic of Korea) grant 2022R1A4A1018904(Lee) and KIAS individual grant MG089101(Oh).

\end{document}